\documentclass[a4paper, 12pt]{amsart}
\usepackage{amsmath}
\usepackage{amsthm}
\usepackage{amssymb}
\usepackage{array}
\newtheorem{thm}{Theorem}

\newtheorem{prop}[thm]{Proposition}
\newtheorem{lemma}[thm]{Lemma}
\newtheorem{remark}[thm]{Remark}

\def\Z{\mathbb Z}

\def\R{\mathbb R}

\def\mod{\operatorname{mod}}

\def\O{\operatorname{O}}
\def\o{\operatorname{o}}
\def\log{\operatorname{log}}

\def\mod{\rm mod \ }

\title{Distribution of gaps of eigenangels of Hecke operators}
\begin{document}
\title[Effective joint dstribution of eigenvalues \dots]{Effective joint distribution of eigenvalues of Hecke operators}
\author[Sudhir Pujahari]{Sudhir Pujahari}
\address{Sudhir Pujahari, Harish-Chandra Research Institute (HBNI), Chatnag Road, Jhunsi, Allahabad - 211019, Uttar Pradesh, India}
\email{sudhirpujahari@hri.ac.in}
\maketitle

\begin{abstract}
In 1997, Serre proved that the eigenvalues of normalised $p$-th Hecke operator $T^{'}_p$ acting on the space of cusp forms of weight $k$ and level $N$ are equidistributed in $[-2,2]$ with respect to a measure that converge to the Sato-Tate measure, whenever $N+k \to \infty$. In 2009, Murty and Sinha proved the effective version of Serre's theorem.  In 2011, using Kuznetsov trace formula, Lau and Wang derived the effective joint distribution of eigenvalues of normalized Hecke operators acting on the space of primitive cusp forms of weight $k$ and level $1$.  In this paper, we extend the result of Lau and Wang to space of cusp forms of higher level. Here we use Eichler-Selberg trace formula instead of Kuznetsov trace formula to deduce our result.
\end{abstract}

\section{Introduction}
Let $S(N,k)$ be the space of all  holomorphic cusp forms of weight $k$ with respect to $\varGamma_0(N).$
For any positive integer $n$, let $T_n(N,k)$ be the $n$-th Hecke operator acting on $S(N,k)$. 
Let $s(N,k)$ denote the dimension of the space $S(N,k)$.  For 
a positive integer $n \geq 1$, 
let $$
\{\lambda_{i}(n)\}, 1 \leq i \leq s(N,k)
$$
  denote the eigenvalues of $T_n$,
counted with multiplicity.
 For any positive 
integer $n$, let 
$$ T_n^{'}:= \frac{T_n}{n^{\frac{k-1}{2}}} $$
 be the normalized Hecke operator acting on $S(N,k)$ 
with eigenvalues 
$$ \left \{ a_i(n)=\frac{\lambda_i(n)}{n^{\frac{k-1}{2}}}, 1 \leq i \leq s(N,k)\right\},$$ counted with multiplicity.
By the celebrated theorem of Deligne~\cite{Deligne},
which proves the famous  Ramanujan conjecture, we know that for any prime $p$, such that $p$ coprime to $N$, 
the eigenvalue of $T_p^{'}$ lies in the interval
$[-2,2]$.  The recently proved Sato-Tate conjecture by Barnet-Lamb, Geraghty, Harris and Taylor~\cite{BL},~\cite{LMR} and~\cite{MNR} says that
if $a_i(p), 1\leq i \leq r$ is a $p$-th normalized Hecke eigenvalue, then the family 
$\{a_i(p)\}$ is equidistributed in $[-2,2]$ as  $p \to \infty$ with respect to the 
Sato-Tate measure
$$d\mu_{\infty}=\frac{1}{2 \pi} \sqrt {4-x^2}dx.$$
More precisely, the Sato-Tate conjecture states that
for any continuous function $\phi: \mathbb{R} \rightarrow \mathbb{R}$, positive integer $V$ and any interval
$[\alpha, \beta] \subset [-2,2]$
$$\lim_{V \to \infty} \frac{1}{V} \sum_{p \leq V} \phi \left(\frac{a_i(p)}{p^{\frac{k-1}{2}}} \right)= 
\int_{\alpha}^{\beta} \phi d \mu_{\infty}.$$

In 1997, Serre~\cite{Serre} studied the ``vertical" Sato-Tate conjecture by fixing a prime $p$ and varying $N$ and $k.$  In particular, 
he proved the following theorem:
\begin{thm}
Let $p$ be a prime number.  Let $\{(N,k)\}$ be a sequence of pairs of positive integers such that $k$ is even and $p$ is coprime to $N$ and $N+k \to \infty.$   Then the family of eigenvalues of the normalized $p$-th Hecke operator 
$$T_{p}^{'}(N,k)= \frac{T_{p}(N,k)}{p^{\frac{k-1}{2}}}$$
is equidistributed in the interval $\Omega=[-2,2]$ with respect to the measure
$$\mu_p:=\frac{p+1}{\pi}\frac{\sqrt{1-\frac{x^2}{4}}}{(p^{\frac{1}{2}}+p^{-\frac{1}{2}})^2-x^2}dx$$
\end{thm}
\begin{remark}
Also in 1997, Conrey, Duke and Farmer~\cite{CDF}  studied a special case of above result by fixing $N=1.$ 
\end{remark}
In 2009, Murty and Sinha~\cite{MS} investigated the effective/quantitative version of Serre's results, 
in which they give explicit estimate on the rate of convergence.  They proved the following theorem
\begin{thm}
Let p be a fixed prime.  Let $\{(N,k)\}$ be a sequence of pairs of positive integers such that $k$ is even, $p$ is
coprime to $N$.  For an interval $[\alpha,\beta] \subset [-2,2],$ 
$$\frac{1}{s(N,k)} \sharp \left \{1 \leq i \leq s(N,k):a_{i}(p) \in [\alpha,\beta] \right \}
=\int_{\alpha}^{\beta}\mu_{p} +\O \left(\frac{\log \,p}{\log \,kN}\right).$$
\end{thm}
As a continuation of their paper, in 2010, Murty and Sinha~\cite{MS} proved a quantitative equidistribution theorem 
for the eigenvalues of Hecke operators acting on the space $S^{new}(N,k)$.
Recently Lau and Wang~\cite{LY} computed the rate of convergence in Sarnak's~\cite{Sarnak} result using the Kuznetsov trace formula.  Indeed they proved the joint distribution of eigenvalues of the Hecke operators quantitatively for primitive Maass forms of level 1 and stated that the same hold true for primitive holomorphic cusp forms.  More precisely, they proved the following theorem:
\begin{thm}
Let $p_1,p_2,,...,$ and  $p_r$ be distinct primes.  Let $k$ be a positive even integer such that 
$r \log \,(p_1p_2 ,..., p_r) \leq \delta \log \, k$, for some small absolutely constant $\delta$.  
Let $a_i^{'}(p_i)$ be the eigenvalues of normalized Hecke operators $T_{p_i}^{'}$ acting on $S^{new}(1,k)$.  For any
$I=[\alpha_n,\beta_n]^r \subset [-2,2]^r$  
\begin{eqnarray*}
&& \frac{ \sharp \left\{1 \leq n \leq s(1,k): \left( a_{1}^{(n)}(p_1),\ldots ,  a_{r}^{(n)}(p_r)  \right) \in I \right\}} {s^{new}(1,k)} \\
&& =  \int_I \prod_{n=1}^rd \mu_{p_n} 
  + \O \left(\frac{r\log \,(p_1p_2 \dots p_r)}{\log \,k} \right),
\end{eqnarray*}
where $\displaystyle d\mu_{p_n}=\frac{p_n+1}{\pi}\frac{\sqrt{1-\frac{x^2}{4}}}{\left({p_n}^{\frac{1}{2}}+{p_n}^{-\frac{1}{2}}\right)^2-x^2}dx.$
\end{thm}
They have remarked that their methods do work for primitive forms in higher level.  In this paper we extend their result to the cusp forms of any level $N$ using ``Eichler-Selberg trace formula".
Precisely, we prove the following theorem:

\begin{thm}\label{T1}
Let $p_1,p_2,...,$  and $p_r$ be distinct primes.  Let $k$ be  positive even integer such that 
$r\log \,(p_1p_2 ,..., p_r) \leq \delta \log  k$, for some small absolutely constant $\delta$.  
For $1 \leq i \leq r,$ let $a_i^{(n)}(p_i), 1\leq n \leq s(N,k)$ be the eigenvalues of normalized Hecke operators $T_{p_i}^{'}$ acting on $S(N,k)$.  For any
$I=[\alpha_n,\beta_n]^r \subset [-2,2]^r$  
\begin{eqnarray*}
&&\frac{\sharp \left \{1\leq n \leq s(N,k): \left(a_{1}^{(n)}(p_1), \ldots , a_{r}^{(n)}(p_r) \right) 
\in I \right \}} {s(N,k)} \\
&& =\int_I \prod_{n=1}^rd \mu_{p_n} + \O \left(\frac{r\log \,(p_1p_2 \cdots p_r)}{\log \,(kN)} \right),
\end{eqnarray*}
where $\displaystyle d\mu_{p_n}=\frac{p_n+1}{\pi}\frac{\sqrt{1-\frac{x^2}{4}}}{({p_n}^{\frac{1}{2}}+{p_n}^{-\frac{1}{2}})^2-x^2}dx$
and the implied constant is effectively computable.
\end{thm}

\section{ Equidistribution and its extension}\label{S2}

A sequence of real numbers $\{x_n\}_{n=1}^{\infty}$ is said to be uniformly distributed or (equidistributed) (\mod 1) if
for any interval $[\alpha,\beta] \subset [0,1],$ we have
$$
\lim_{V \to \infty} \frac{1}{V} \sharp \{n \leq V: x_n \in [\alpha, \beta]\}=\beta-\alpha.
$$
In 1916, Weyl~\cite{Weyl} proved the following criterion for uniform distribution (\mod 1).  
A sequence $\{x_n\}_{n=1}^{\infty}$ is uniformly distributed if and only if for any integer $m \neq 0$
$$
\sum_{n \leq V} e(mx_n)=\o(V) \ \text{as $V \to \infty$}.
$$
Since the set of trigonometric polynomials is dense in $C^{1}[0,1]$, the above criterion of Weyl is equivalent to the 
assertion that, for any continuous function $f: \R \rightarrow \R$, we have 
$$
\lim_{V \to \infty} \frac{1}{V} \sum_{n \leq V}f(x_n)= \int_{0}^1f(t)dt.
$$
A sequence of tuples $\left \{(x_{1}^{(n)},x_{2}^{(n)}, \ldots ,x_{r}^{(n)})\right \}$ in ${\R}^r$ is said to be uniformly distributed or
(equidistributed) (\mod 1) if for every $I=\prod_{n=1}^r [\alpha_n, \beta_n] \subset [0,1]^r$, we have 
$$
\lim_{V \to \infty} \frac{\sharp \{n \leq V:(x_{1}^{(n)},x_{2}^{(n)}, \ldots , x_{r}^{(n)}) \in I  \}}{V}
= \mu(I),
$$
where $\mu$ is the usual Lebesgue measure on $\R^{r}$.  In the same paper Weyl~\cite{Weyl}  extended his criterion of 
equidistribution to the higher dimension as follows:  \\
A sequence of tuples $\left\{(x_{1}^{(n)},x_{2}^{(n)}, \ldots , x_{r}^{(n)})\right\}$ in $\mathbb{R}^r$ is uniformly distributed if and only 
if for any integers $m_1,m_2,,...,, m_r \neq 0$
$$
\sum_{n=1}^{V}e \left(\sum_{i=1}^r m_ix_n \right) = \o(V) \
 \mbox{ as }  V \to \infty.
 $$
Note that as in the one variable  case, the set of trigonometric polynomials is also dense in $C([0,1]^r),$ 
the above criterion is equivalent to the following statement: \\
A sequence of tuples $\left\{(x_{1}^{(n)},x_{2}^{(n)}, \ldots , x_{r}^{(n)})\right\}$ in ${\R}^r$ is uniformly distributed if and only 
if for any continuous function $f: {\R}^r \rightarrow \R,$
$$
\lim_{V \to \infty} \frac{1}{V}\sum_{n \leq V}f(x_{1}^{(n)},x_{2}^{(n)}, \ldots ,x_{r}^{(n)}) =\int_{I}f(x_{1},x_{2}, \ldots ,x_{r}) dx_{1}dx_{2} \cdots dx_{r}.
$$

Now we define the set equidistribution as 
follows:    \\
A sequence of finite multi sets $A_n$ with $\sharp A_n \to \infty$ is said to be set equidistributed (\mod 1) with respect to a 
probability measure $\mu$ if for every continuous function $f \in C^{1}[0,1]$, we have 
$$
\lim_{n \to \infty } \frac{1}{\sharp A_n} \sum_{t \in A_n}f(t)= \int_{0}^1 f(x)d \mu.
$$
The criterion of Schoenberg and Wiener says that
the sequence $\{A_n\}$ is set equidistributed with respect to some positive continuous measure if and only if 
the Weyl limit
$$c_m := \lim_{n \to \infty} \frac{1}{\sharp A_n} \sum_{t \in A_n}e(mt)$$
exists and 
$$
\sum_{m=1}^{V}|c_m|^2 = \o(V).
$$
For our purpose, let us define the set equidistribution in higher dimension.
A tuples of finite multi set say $ \Omega =(A_{1}^{(n)},A_{2}^{(n)}, \ldots , A_{r}^{(n)})$ with   
$\sharp A_{i}^{(n)} \to \infty$ as $n \to \infty$ for all  $ i=1,2,,...,,r $ is said to be set a set equidistributed (\mod 1) with respect to a probability measure $\mu$ if for every continuous function 
$f \in C^{1}([0,1]^r)$, and $I=[0,1]^r$
we have 
$$
\lim_{V \to \infty} \frac{1}{V} \sum_{n \leq V} f(x_{1},x_{2}, \ldots ,x_{r}) =\int_{I}f(x_{1},x_{2}, \ldots ,x_{r}) dx_{1}dx_{2} \cdots dx_{r}.
$$
With this generalization, we define the Weyl limit as  
$$
C_{\underline m}:= \lim_{n \to \infty} \frac{1}{\prod_{i=1}^{r}\sharp A_n} \sum_{(t_1,t_2, \ldots ,t_r) \in \Omega}e(m_1t_1+m_2t_2+\cdots+m_rt_r).
$$


\section{Eichler Selberg Trace Formula and its Estimations}

In this section, we use Eichler-Selberg trace formula, as one of our important tool to prove the main theorem, which 
is a formula for the trace of $T_n$ acting on $S(N,k) $ in terms of class number of binary quadratic forms and few others arithmetic 
functions.  We follow the presentation of~\cite{MS}.  For a non-negative integer $\vartriangle \equiv 0,1 $ (\mod 4), let $B(\vartriangle)$ be the set of all positive definite binary quadratic forms
with discriminant $\vartriangle$.  That is 
$$B(\vartriangle)= \{ax^2+bxy+cy^2 : a,b,c \in \Z, a>0, b^2-4ac=\vartriangle\}.$$
We denote the set of primitive forms by
$$b(\vartriangle)= \{f(x,y) \in B(\vartriangle): gcd(a,b,c)=1\}.$$
Now we define an action of the full modular group $SL_2(\Z)$ on $B(\vartriangle)$ as follows:
$$f(x,y) \begin{pmatrix}
\alpha & \beta  \\
\gamma & \delta 
\end{pmatrix}:= f(\alpha x+ \beta y, \gamma x + \delta y).$$
Note that the above action takes primitive forms to primitive forms.  We know that the above action has finitely many orbits.  We define $h(\vartriangle)$ to be the number of orbits of
$b(\vartriangle)$.  
Let $h_w$ be defined as follows:
\begin{eqnarray*}
 h_w(-3) & = & \frac{1}{3},  \\
 h_w(-4) & = & \frac{1}{2},  \\
 h_w(\vartriangle) & = & h(\vartriangle) \ \text {for $\vartriangle < -4$}.
\end{eqnarray*}
We define some arithmetical functions which are going to useful to state the Eichler-Selberg trace formula.

Let \begin{eqnarray*}
 A_1(n)  =
 \left\{
 \begin{array}{l l}
  n^{\frac{k}{2}-1}\left(\frac{k-1}{12} \right) \psi(N) &  \text{if $n$ is a square}   \\
  0 & \text{otherwise,}
  \end{array}
  \right.
\end{eqnarray*}
where $\displaystyle\psi(N) = N \prod_{p | N}\left(1 + \frac{1}{p}\right)$. Let 
$$
A_2(n)=- \frac{1}{2} \sum_{{t \in \Z} \atop{ t^2 < 4n}} \frac{\varrho^{k-1}-{\bar \varrho}^{k-1}}{\varrho-\bar \varrho}
\sum_{g} h_w \left( \frac{t^2-4n}{g^2} \right) \mu(t,g,n),
$$
where $\varrho$ and $\bar \varrho$ are the zeros of the polynomial $x^2-tx+n$, the inner sum runs over positive divisors
$g$ of $\frac{(t^2-4n)}{g^2} \in \Z $ is congruent to $0$ or $1 \ (mod \ 4)$ and 
$\mu(t,g,n)$ is given by 
$$\mu(t,g,n)= \frac{\psi (N)}{\psi (\frac{N}{N_g})}M(t,n,NN_g),$$ 
with  $N_g=gcd(N,g)$ and $M(t,n,k)$ denote the number of solutions of the congruence $x^2-tx+n \equiv 0 \ \ \pmod{ K}$. Now, we let 
\begin{equation}
A_3(n)=- \sum_{{d|n} \atop{ 0  <d \leq \sqrt{n}}}d^{k-1} \sum_{c|N} \phi \left( gcd \left(c, \frac {N}{c}\right)\right),
\end{equation}
where  $\phi$ denotes the Euler's function and in the first summation, if there is a contribution from the term  $d= \sqrt{n}$,
it should be multiplied by $\frac{1}{2}.$  In the inner sum, we also need the condition that $gcd \ (c, \frac{N}{c})$ divides 
$gcd(N,\frac{n}{d}-d).$ Finally, let 
\begin{eqnarray*}
A_4(n)= 
\displaystyle\left\{
\begin{array} {l l}
{\sum \atop {{t|n \atop t > 0}}}t &    \text{if $k=0$}  \\
0 &   \text {otherwise}.
\end{array}
\right.
\end{eqnarray*}
\begin{thm}
For any positive integer $n$, let $Tr(T_n)$ be the trace of $T_n$ acting on $S(N,k).$ Then, we have, 
$$Tr (T_n)=A_1(n)+A_2(n)+A_3(n)+A_4(n).$$
\end{thm}

We use the above Eichler-Selberg trace formula to prove the following Proposition:
\begin{prop}\label{P1}
For any positive integers $m_1, m_2, \ldots ,m_r$ we have, 
\begin{align*}
|Tr \ (T_{p_1^{m_1} \cdots p_r^{m_r}})|   
 &\ll p_1^{\frac{3m_1}{2}}\cdots p_r^{\frac{3m_r}{2}}(m_1 \cdots m_r)d(N)\sqrt{N}\log \,(4(p_1^{m_1} \cdots p_r^{m_r})). 
 \end{align*}
 
\end{prop}
\begin{proof} Consider 
\begin{eqnarray*}
Tr ( T^{'}_{p_1^{m_1} \cdots p_r^{m_r}}) & = &  B_1(m_1m_2 \cdots m_r)+B_2(m_1m_2 \cdots m_r)+B_3(m_1m_2 \cdots m_r) \\
&& +B_4(m_1m_2 \cdots m_r),
\end{eqnarray*}
where $B_i(m_1m_2 \cdots m_r)= \displaystyle\frac{A_i(p_1^{m_1} \cdots p_r^{m_r})}{(p_1^{m_1} \cdots p_r^{m_r})^{\frac{k-1}{2}}}$, for all  $i=1,2,3,4.$  
We use the estimates of~\cite{MS} and get the estimates for each  $B_i(m_1m_2 \cdots m_r),$ for all $i=1,2,3,4.$  First, we consider $i=1$
\begin{align*}
B_1(p_1^{m_1} \cdots p_r^{m_r}) = & \frac{A_1(p_1^{m_1} \cdots p_r^{m_r})}{(p_1^{m_1} \cdots p_r^{m_r})^{\frac{k-1}{2}}}  \\
= &
\left\{
\begin{array}{l l}
   \frac{\left({p_1}^{-\frac{m_1}{2}} \cdots {p_r}^{-\frac{m_r}{2}} \right) \frac{k-1}{12}}{(p_1^{m_1} \cdots p_r^{m_r})^{\frac{k-1}{2}}}  & 
  \text{if $m_1,m_2,,...,, m_r$ are even,} \\
    0 &      \text{ otherwise.}  
    \end{array}
\right. \\
  = &
 \left\{
 \begin{array}{l l}
  (p_1^{-\frac{m_1}{2}}\cdots p_r^{-\frac{m_r}{2}})\frac{k-1}{12} &  \text{if $m_1,m_2,...,m_r$ are even,} \\
  0 &    \text{ otherwise.}  
  \end{array}
  \right.
  \end{align*}
  Now, we shall consider $i=2$ as follows.
\begin{eqnarray*}
   B_2(p_1^{m_1} \cdots p_r^{m_r}) & = & \frac{A_2(p_1^{m_1} \cdots p_r^{m_r})}{(p_1^{m_1} \cdots p_r^{m_r})^{\frac{k-1}{2}}} \\
 & = & \frac{8e}{\log \,2}(2)p_1^{\frac{3m_1}{2}}\cdots p_r^{\frac{3m_r}{2}}2^{\nu(N)}\log \,4(p_1^{m_1} \cdots p_r^{m_r}). 
  \end{eqnarray*}
  For $i=3$, we get, 
  \begin{align*}
 B_3(p_1^{m_1} \cdots p_r^{m_r})  = & \frac{A_3(p_1^{m_1} \cdots   p_r^{m_r})d(N)\sqrt{N}}{(p_1^{m_1} \cdots p_r^{m_r})^{\frac{k-1}{2}}} \\ 
  \leq & \frac{d(p_1^{m_1} \cdots p_r^{m_r})(p_1^{m_1} \cdots p_r^{m_r})^{\frac{k-1}{2}}d(N)\sqrt{N}}{(p_1^{m_1} \cdots p_r^{m_r})^{\frac{k-1}{2}}}  \\
  \leq & (m_1+1)\cdots(m_r+1)d(N)\sqrt{N}. 
  \end{align*}
  Finally, for $i=4$, we get, 
  \begin{align*}
 B_4(p_1^{m_1} \cdots p_r^{m_r})  = & \frac{A_4(p_1^{m_1} \cdots p_r^{m_r})}{(p_1^{m_1} \cdots p_r^{m_r})^{\frac{k-1}{2}}}  
  = 
 \left\{
 \begin{array}{l l} 
 \frac{d(p_1^{m_1} \cdots p_r^{m_r})(p_1^{m_1} \cdots p_r^{m_r})}{(p_1^{m_1} \cdots p_r^{m_r})^{\frac{k-1}{2}}} & \text{if $k=2$}  \\
  0 &  \text{otherwise.}   
  \end{array}
  \right.  \\
  = & 
 \left\{
 \begin{array}{l l}
{p_1}^{\frac{m_1}{2}}\cdots{p_r}^{\frac{m_r}{2}}(m_1+1)\cdots(m_r+1) & \text{if $k=2$} \\
 0 &  \text{otherwise.}
 \end{array}
 \right.
 \end{align*}
 Combining all the estimates above, we get 
\begin{eqnarray*}
 &&|Tr \ (T_{p_1^{m_1} \cdots p_r^{m_r}})| \\
 && \ll   (p_1^{-\frac{m_1}{2}}\cdots p_r^{-\frac{m_r}{2}})\frac{k-1}{12} 
 + {p_1}^{\frac{3m_1}{2}}\cdots{p_2}^{\frac{3m_r}{2}}\log \,4(p_1^{m_1} \cdots p_r^{m_r}) \\
 && +  (m_1+1)\cdots(m_r+1)d(N)\sqrt{N}+ p_1^{m_1} \cdots p_r^{m_r}(m_1+1)\cdots(m_r+1)   \\
 && \ll    {p_1}^{\frac{3m_1}{2}}\cdots{p_r}^{\frac{3m_r}{2}}\log \,(4(p_1^{m_1} \cdots p_r^{m_r}))d(N)\sqrt{N}(m_1m_2 \cdots m_r),
  \end{eqnarray*}
 where $d(N)$ is the divisor function.
\end{proof}

 \section{The measure $\mu$}

Following the definition from Section \ref{S2},
for non-negative integers $m_1,m_2,\ldots, m_r$, we get,
\begin{eqnarray*}
C_{\underline m} & := &\lim_{s(N,k) \to \infty} \frac{1}{(2 s(N,k))^r} \sum_{n\atop{1 \leq i \leq n}}e( \pm m_1 \theta_{1}^{(n)}(p_1) \pm
\cdots\pm m_r \theta_{r}^{(n)}(p_r) )  \\
& = & \prod_{i=1}^r \left(\lim_{s(N,k) \to \infty} \frac{1}{(2s(N,k))^r}\sum_{i_n}e(\pm m_i \theta_{i}^{(n)}(p_i)) \right)  \\
& = & \prod_{i=1}^r c_{m_i},
\end{eqnarray*}

where $c_{m_i} $ are the $m_i$-th Weyl limit of the family $\left\{\frac{\pm \theta_{i}^{(n)}}{2 \pi}\right \}$ and from Theorem 18 of~\cite{MS}, we have
\begin{eqnarray*}
c_{m_i}=
\left \{
\begin{array}{l l}
1 & \text{if } m_i=0, \\
\frac{1}{2} \left( \frac{1}{{p_i}^{m_i}}-\frac{1}{{p_i}^{m_i-2}}\right) & \text{if } m_i \text{ is even}, \\
0 &  \text{otherwise.}
\end{array}
\right.
\end{eqnarray*}
Note that if $m_1,m_2,\ldots, m_r$ are all zero then $C_{0,\ldots,0}=1$.
Define the measure 
$$\mu=F(-x_1,\ldots,-x_r) dx_1 \cdots dx_r,$$
where 
\begin{eqnarray*}
F(x_1,\ldots,x_r) 
& = & \sum_{m_1=-\infty}^{\infty} \cdots \sum_{m_r=-\infty}^{\infty}c_{m_1} \cdots c_{m_r} e(m_1x_1) \cdots e(m_rx_r)  \\
& = & F(x_1) \cdots F(x_r)\\
& = &  \frac{4(p_1+1) \sin^2 2 \pi x_1}{({p_1}^\frac{1}{2}+{p_1}^{-\frac{1}{2}})^2-4\cos^2 2 \pi x_1} \cdots
\frac{4(p_r+1) \sin^2 2 \pi x_r}{({p_r}^\frac{1}{2}+{p_r}^{-\frac{1}{2}})^2-4\cos^2 2 \pi x_r}. 
\end{eqnarray*}
The above $F(x_1,\ldots,,x_r)dx_1 \cdots dx_r$ determines a measure on $[0,1]^r$ and is the distribution function 
for the tuples of numbers 
$$
(x_{i_1}, \ldots ,x_{i_r})=\left(\pm \frac{\theta_{1}^{(n)}(p_1)}{2 \pi},\ldots,\pm \frac{\theta_{r}^{(n)}(p_r)}{2 \pi}\right).
$$
The measure giving the distribution of 
$(\cos \theta_{1}^{(n)}(p_1), \ldots ,\cos \theta_{r}^{(n)}(p_r))$ is therefore 
\begin{eqnarray*}
& = & F\left(\frac{\cos^{-1}x_1}{2 \pi}, \ldots ,\frac{\cos^{-1}x_r}{2 \pi}\right)d\left(\frac{\cos^{-1}x_1}{2 \pi}\right) \cdots d\left(\frac{\cos^{-1}x_r}{2 \pi}\right)  \\
& = & F\left(\frac{\cos^{-1}x_1}{2 \pi} \right) d \left(\frac{\cos^{-1}x_1}{2 \pi} \right) \cdots
F \left(\frac{\cos^{-1}x_r}{2 \pi} \right) d \left(\frac{\cos^{-1}x_r}{2 \pi} \right) \\ 
& = & \frac{2(p+1)}{\pi} \frac{\sqrt{1-{x_1}^2}}{(p^{\frac{1}{2}}+p^{-\frac{1}{2}})^2-4{x_1}^2}dx_1 \cdots
\frac{2(p+1)}{\pi} \frac{\sqrt{1-{x_r}^2}}{(p^{\frac{1}{2}}+p^{-\frac{1}{2}})^2-4{x_r}^2}dx_r.
\end{eqnarray*}
Thus, the distribution of the tuples of numbers 
$$(2\cos \theta_{1}^{(n)}(p_1)\cdots 2\cos \theta_{r}^{(n)}(p_r))$$ is given by
$\prod_{i=1}^r\mu_{p_i} $ after an 
easy change of variable.

\section{Chebychev polynomials and the trace of Hecke operators}

For any integer $n \geq 0$, the $n$-th Chebychev polynomial $X_n(x)$ is defined as follows:
$$
X_n(x)=\frac{\sin (n+1)\theta}{\sin \theta}, \ \ \text{where } x=2 \cos \theta .$$
Serre proved the following result in~\cite{Serre}.

\begin{lemma}
We have $ T^{'}_{p^m}=X_m(T_p).$
\end{lemma}

From~\cite[page 697]{MS},  when $m_i=1,$ we have, 
$$
\sum_{n=1}^{s(N,k)} 2 \cos \theta_{i}^{(n)}(p_i)= Tr(T^{'}_{p_i}).
$$
Since,  for all integers $m\geq 2$, 
$$\sum_{n=1}^{s(N,k)} 2 \cos m \theta = X_m(2 \cos \theta)-X_{m-2}(2 \cos \theta), 
$$
we have for $m_i \geq 2,$
\begin{equation}\label{E13}
\sum_{n=1}^{s(N,k)} 2 \cos \theta_{i}^{(n)}(p_i)=
Tr(T^{'}_{p_i^{m_i}})-Tr(T^{'}_{p_i^{m_i-2}}).
\end{equation}
Now we prove the following theorem,

\begin{thm}
For all non zero positive integers $m_1,m_2,\ldots,m_r$ consider $\underline{m}=(m_1, m_2,\ldots, m_r),$ the Weyl limits $C_{\underline m}$ are given by 
\begin{eqnarray*}
C_{\underline m}  = 
\left \{ 
\begin{array}{l l}
1 & \text{if $m_1= \cdots =m_r=0$,} \\
\prod_{i=1}^r \left(\frac{1}{{p_i}^{\frac{m_i}{2}}}-\frac{1}{{p_i}^{\frac{m_i-2}{2}}} \right) &   \text{if $m_i,\  i=1,2,..,r$ are even,}  \\
 0 & \text{otherwise.}    
 \end{array}
 \right.
\end{eqnarray*}
Moreover, 
\begin{eqnarray*}
&& \left|\prod_{i=1}^r \sum_{n=1}^{s(N,k)}2 \cos \ m_i \theta_{i}^{(n)}(p_i)-s(N,k) C_{\underline m}\right| \\
&& \ll  
p_1^{\frac{3m_1}{2}} \cdots p_r^{\frac{3m_r}{2}}(m_1 \cdots m_r)\log \,(4(p_1^{m_1} \cdots p_r^{m_r})).
\end{eqnarray*}
\end{thm}
\begin{proof}
For any integers $m_1,m_2,\ldots,m_r \geq 2$, using (\ref{E13}), we have
 \begin{eqnarray}\label{E23}
&& \prod_{i=1}^{r} \sum_{n=1}^{s(N,k)} \{2 \cos m_i\theta_{i}^{(n)}(p_i)\} \\
   &= & \prod_{i=1}^{r}\left(X_{m_i}(2\cos  \theta_{i}^{(n)}(p_i))-X_{m_i-2}(2\cos  \theta_{i}^{(n)}(p_i))\right) \nonumber\\ 
& = &\sum_{b_1,b_2,\ldots,b_r \in Z} \prod_{i=1}^{r} X_{b_i}(p_i)(2\cos \theta_{i}^{(n)}(p_i)), \nonumber 
 \end{eqnarray}
 where $Z=\{m_1,m_2,\ldots,m_r,m_1-2,m_2-2,\ldots,m_r-2\}$.
We know that if a linear operator $T$ is diagonalizable and $\lambda$ is an eigenvalue of $T$, then, for any polynomial $P(x)$,  the eigenvalue of $P(T)$ is $P(\lambda)$.
 Since the Hecke operators $T_m$  and $T_n$ commutes with each other, there exists an ordered basis such that every Hecke operator can be represented by a diagonal matrix with respect to the basis. 
 Using all the above facts, (\ref{E23}) equals
 \begin{eqnarray}\label{E29}
&& \nonumber \sum_{b_1,b_2,\ldots,b_r \in Z} Tr ( T^{'}_{p_1^{b_1}\cdots p_r^{b_r}}) \\ \nonumber
  &\ll& \sum_{b_1,b_2,\ldots,b_r \in Z}(p_1^{b_1}\cdots p_r^{b_r})\left(\frac{k-1}{12}\right)^r \\ \nonumber
 && +p_1^{\frac{3m_1}{2}}\cdots p_r^{\frac{3m_r}{2}}\log \,4(p_1^{m_1} \cdots p_r^{m_r})(m_1m_2 \cdots m_r)\\ \nonumber
 &\ll & \left(\frac{k-1}{12}\right)^r \prod_{n=1}^r\left( \frac{1}{{p_n}^{\frac{m}{2}}}-\frac{1}{{p_n}^{\frac{m-2}{2}}} \right) \\  \nonumber
 && +p_1^{\frac{3m_1}{2}} \cdots p_r^{\frac{3m_r}{2}}\log \,4(p_1^{m_1} \cdots p_r^{m_r})(m_1m_2 \cdots m_r) \\  
& \ll & \left(\frac{k-1}{12}\right)^r \prod_{n=1}^r\left( \frac{1}{{p_n}^{\frac{m_n}{2}}}
-\frac{1}{{p_n}^{\frac{m_n-2}{2}}} \right)|Tr ( T^{'}_{{p_1}^{m_1}{p_2}^{m_2}\cdots{p_r}^{m_r}})|.
\end{eqnarray}
Now
\begin{eqnarray*}
&& \left|\prod_{i=1}^r  \sum_{n=1}^{s(N,k)}2 \cos \ m_i \theta_{i}^{(n)}(p_i)- s(N,k) C_{\underline m}\right| \\
& = & \left|\sum_{j=1}^{4} B_j(m_1 \cdots m_r)-s(N,k) C_{\underline m} \right|  \\
& \leq & \left|B_1-s(N,k) C_{\underline m}\right|+ \left|\sum_{j=2}^{4}B_j(m_1 \cdots m_r)\right|.
\end{eqnarray*}
 Using (\ref{E29}) and the fact that $s(N,k) C_{\underline m} \ \text{behaves like}  \ B_1(m_1 \cdots m_r), $ we have   
 $$\left|\prod_{i=1}^r \sum_{n=1}^{s(N,k)}2 \cos \ m_i \theta_{i}^{(n)}(p_i)-s(N,k) C_{\underline m} \right| $$
 $$ \ll  p_1^{\frac{3m_1}{2}}\cdots p_r^{\frac{3m_r}{2}}(m_1 \cdots m_r)\log \,(4(p_1^{m_1} \cdots p_r^{m_r})). $$
 \end{proof}

 \section{Effective versions}

 To prove the effective result, we use the following variant of the Erd\"{o}s-Tur\'{a}n inequality that can be found in~\cite[Proposition 7.1]{LLW}.
 
 \begin{thm}\label{ET}
For any $S=\prod_{n=1}^{r}[\alpha_n,\beta_n] \subset [0, \frac{1}{2}]^r$ and a sequence of tuples of number $\{(x_{m_1},\ldots, x_{m_r})\} \in  [0, \frac{1}{2}]^r,$  we define
$$
N_S(V):= \sharp \{n \leq V:(x_{m_1},\ldots,x_{m_r}) \in S \}
$$ 
and 
$$
D_S(V):=|N_S(V)-V \mu(S)|.
$$
 For
$I=\prod_{n=1}^{r}[\alpha_n,\beta_n],$
we have 
\begin{align*}
D_I(V) \leq & \sum_{m_1,m_2,\ldots,m_r \in ([-m,m] \bigcap {\Z})^r}w(m_1,\ldots,m_r)\Delta((m_1,\ldots,m_r),V)\\
&+10 \frac{2}{M+1}\sum_{1 \leq m \leq M} {\rm max \atop{1 \leq t \leq r}} \Delta_{t}(m,V)
+ 12 \Vert F\Vert \frac{2V}{M+1} 
\end{align*}
for any integers $V,M \geq 1$, where
$$
\Delta_t(m,V)=|\sum_{n \leq V}\cos(2 \pi m x_{i}) -Vc_{m}|, \  \ \ \text{(for $m \in \Z$)},
$$
$$
\left|\Delta((m_1,\ldots,m_r),V)\right| =\left|\sum_{m_1,\ldots, m_r \leq V}\cos(2 \pi m_1 x_1) \cdots \cos(2 \pi m_r x_r)\right|,   
$$
for $(m_1,\ldots,m_r) \in \Z^r$ and
$$
w(m_1,\ldots, m_r)=\left((2 \pi)^r \prod_{t=1}^r \rm min \left(\frac{1}{\pi|m_t|},\beta_t-\alpha_t \right)+\frac{2}{M+1}\right),$$
and 
$$\Vert F\Vert=\max_{1 \leq t \leq r} \Vert F_t\Vert_{\infty}+\prod_{t=1}^r \Vert F_t\Vert_{\infty}$$
with $\Vert F_t\Vert_{\infty}= \max_{x \in [0,1]} |F_t(x)|.$
\end{thm}

\bigskip

\noindent\textbf{Proof of Theorem \ref{T1}.} \\
Choose $\theta_{i}^{(n)}(p_i) \in [0, \frac{1}{2}]$ such that $2 \cos 2 \pi \theta_{i}^{(n)}(p_i)=a_{i}^{(n)}(p_i).$  Given any sub interval $I \subset [-2,2]$ choose sub interval 
$I^{'} \subset [0, \frac{1}{2}]$ so that $\theta_{i}^{(n)}(p_i) \in I^{'}$ if and only if $a_{i}^{(n)} \in I.$
By Theorem \ref{ET}, we have
\begin{align*}
\left|\sharp \right.\{1 \leq n  \leq (s(N,k)) &: (a_{1}^{(n)}(p_1),\ldots, a_{r}^{(n)}(p_r)) \in I  - V \mu(I)\left. \right|\\
&\ll \frac{s(N,k)r}{M+1}+\left |\prod_{i=1}^r \sum_{n=1}^{s(N,k)}2 \cos \ m_i \theta_{i}^{(n)}(p_i)-s(N,k) C_{\underline m}\right|.
\end{align*}
Using Proposition \ref{P1}, we get
$$\ll \frac{s(N,k)r}{M+1}
+(p_1^{\frac{3m_1}{2}} \cdots p_r^{\frac{3m_r}{2}})(m_1 \cdots m_r)\log \,(4(p_1^{m_1} \cdots p_r^{m_r}))\log \,M^r.$$
Since the main contribution is coming from $\left(p_1^{\frac{3m_1}{2}}\cdots p_r^{\frac{3m_r}{2}} \right)$,  by choosing $\displaystyle M=\frac{\log \,kN}{ \log \,p_1p_2 \cdots p_r}$, we get the required result.   $\hfill\Box$

\bigskip

\noindent\textbf{Acknowledgments:}
The author would like to thank Prof$.$ M. Ram Murty and Dr$.$ Kaneenika Sinha for useful discussions in the earlier version of this manuscript.  I am  also thankful to Prof$.$ R. Thangadurai and Dr$.$ Jaban Meher for their suggestions to improve the presentation of the paper.


\begin{thebibliography}{9999}
\bibitem{BL} T. Barnet-Lamb, D. Geraghty, M. Harris and R. Taylor, \textit{A family of Calabi-Yau varieties and potential automorphy II}, Publ. RIMS Kyoto Univ. 47 (2011), 29--98.

\bibitem{LMR} L. Clozel, M. Harris and R. Taylor, \textit{Automorphy for some l-adic lifts of automorphic mod l Galois representations,} Publ. Math. Inst. Hautes tudes Sci. No. 108 (2008), 11--81.

\bibitem{CDF} J.B. Conrey, W. Duke, D.W. Farmer, \textit{The distribution of the eigenvalues of Hecke operators}, Acta Arith., 78 (4) (1997), 405--409.

\bibitem{Deligne} P. Deligne, \textit{La conjecture de Weil I, IHES Publ}. Math. No. 43 (1974), 273-307.

\bibitem{MNR} M. Harris, N. Shepherd-Barron and R. Taylor, \textit{A family of Calabi-Yau varieties and potential automorphy,} Ann. of Math. (2) 171 (2010), no. 2, 779--813.

\bibitem{LLW} Y.-K. Lau, Charles Li, Y. Wang, \textit{Quantitative analysis of the Satake parameters of $GL_2$ representations with prescribed local representations,} Acta Arithmetica, 164.4 (2014).

\bibitem{LY} Y.-K. Lau and Y. Wang, \textit{Quantitative version of the joint distribution of eigenvalues of the Hecke operators}, J. Number Theory 131 (2011), 2262--2281.


\bibitem{MS} M. R. Murty and K. Sinha, \textit{Effective equidistribution of eigenvalues of Hecke operators},  J. Number Theory, {\ 129}  (2009),  no. 3, 681--714. 

\bibitem{MS2} M. R. Murty, K. Sinha, \textit{Factoring new parts of Jacobians of certain modular curves}, Proc. Amer. Math. Soc. 138 (2010), no. 10, 3481-3494.

\bibitem{Sarnak} P. Sarnak, \textit{Statistical properties of eigenvalues of the Hecke operators,} Analytic Number Theory and Diophantine Problems, Stillwater, 1984, in: Progr. Math., vol. 70, Birkhäuser, Basel, 1987, 321-331.


\bibitem{Serre} J-P Serre, \textit{R\'{e}partition asymptotique des valeurs propres de l'op\'{e}rateur de Hecke $T_p$}, J. Amer. Math. Soc., {10} (1997), no. 1, 75--102. 

\bibitem{Weyl} H. Weyl, \textit{\"{U}ber die Gleichverteilung von Zahlen mod Eins,} Math. Ann, Vol. 77 (1916), no. 3, 313--352.





\end{thebibliography}
\end{document}